\documentclass[10pt, reqno]{amsart}
\usepackage{amsmath, amsthm, amscd, amsfonts, amssymb, graphicx, color}
\usepackage{mathrsfs,amssymb}
\usepackage[bookmarksnumbered, colorlinks, plainpages]{hyperref}
\usepackage{amsmath,amsthm,amssymb, amsfonts, amssymb, graphicx, color}
\usepackage[bookmarksnumbered,colorlinks,plainpages]{hyperref}

\def\authorsaddresses#1{\dedicatory{#1}}
\newtheorem{theorem}{Theorem}[section]

\newtheorem{proposition}[theorem]{Proposition}
\newtheorem{corollary}[theorem]{Corollary}
\theoremstyle{definition}
\newtheorem{definition}[theorem]{Definition}
\newtheorem{example}[theorem]{Example}

\theoremstyle{remark}
\newtheorem{remark}[theorem]{Remark}
\numberwithin{equation}{section}

\begin{document}
\setcounter{page}{1}

\title[]{A Dynamic Residual Measure of Inaccuracy Based on Varextropy for Order Statistics  }

\author[F. Goodarzi]{Faranak Goodarzi}

\authorsaddresses{
Department of Statistics, University of Kashan, Kashan, Iran.\\
f-goodarzi@kashanu.ac.ir}

\begin{abstract}
 In this paper, we introduce a variance-based inaccuracy measure for order statistics based on the concept of varextropy. The proposed measure provides a framework for quantifying the variability of the inaccuracy function between an order statistic and its parent distribution. We establish a transformation property of the proposed measure under strictly increasing differentiable transformations. We establish theoretical upper and lower bounds for this measure. We further develop a dynamic residual version of the proposed inaccuracy measure and establish characterization results for the underlying distribution. Furthermore, we derive bounds for the dynamic residual inaccuracy measures based on extropy and varextropy. Finally, a nonparametric estimator of the proposed measure is developed and applied to real data to assess the goodness-of-fit of candidate distributions and facilitate model selection. 
\end{abstract}

\maketitle


\section{Introduction}
Let $X_1, X_2, \ldots, X_n$ be a collection of independent and identically distributed
nonnegative continuous random variables with cumulative distribution function
$F_X(x)$, probability density function $f_X(x)$, and survival function $\bar{F}_X(x)=1-F_X(x)$, defined on the support $S_X$.
Arranging the sample in ascending order yields the order statistics, denoted by
$X_{1:n}\leq X_{2:n}\leq \cdots \leq X_{n:n}$.
Order statistics constitute a fundamental tool in probability and statistics and have
been extensively investigated in the literature.
Their wide range of applications includes reliability analysis, statistical inference, goodness-of-fit testing, characterization of probability
distributions, material strength analysis, outlier identification, survival analysis,
signal processing, risk assessment, and the study of extreme events.

From the perspective of reliability theory, order statistics provide a natural framework for modeling system lifetimes. In particular, the $i$th order statistic
in a sample of size $n$ corresponds to the lifetime of an $(n-i+1)$-out-of-$n$ system, thereby establishing an important connection between order statistics and coherent system reliability.

Let $X$ and $Y$ be two nonnegative continuous random variables describing the lifetimes of
two systems. Denote their probability density functions by $f(x)$ and $g(x)$, their cumulative
distribution functions by $F(x)$ and $G(x)$, and the corresponding survival functions by
$\bar{F}(x)=1-F(x)$ and $\bar{G}(x)=1-G(x)$, respectively.
 Then, the Shannon differential entropy \cite{Shannon} of $X$ is  given by 

\begin{eqnarray}\label{En}
H(X)
=-\int_{0}^{+\infty}f(x)\log f(x)dx,
\end{eqnarray} 
where $-\log f(X)$ represents the information content of $X$.

It is important to note that distinct random variables may exhibit identical Shannon entropy values. In such cases, a more refined characterization of the underlying distributions can be obtained by considering the variability of their information content. In this context, the variance of the information content, known as the varentropy (or variance of entropy) of a random variable $X$  (see \cite{Frad}), is defined as follows
\begin{align}\label{VEn}
V(X)&=
E[(-\log f(X))^2]-[H(X)]^2\nonumber\\&=
\int_{-\infty}^{+\infty}f(x)[\log f(x)]^2dx-
\left[\int_{-\infty}^{+\infty}f(x)\log f(x)dx\right]^2.
\end{align}  
The uncertainty associated with the $i$th order statistic $X_{i:n}$ can be quantified by Shannon's entropy, which is given by
\begin{align}\label{enos}
H(X_{i:n})
=-\int_{0}^{+\infty} f_{i:n}(x)\log f_{i:n}(x)\,dx.
\end{align}
where $f_{i:n}(x)$ is the probability density function of $X_{i:n}$, namely,
\begin{align}\label{eq1.4}
f_{i:n}(x)=\frac{1}{B(i,n-i+1)}
F^{\,i-1}(x)\bar{F}^{\,n-i}(x)f(x),
\qquad i=1,2,\ldots,n,
\end{align}
with
\begin{align}
B(i, n-i+1)=\frac{\Gamma(i)\Gamma(n-i+1)}
{\Gamma(n+1)},
\end{align}
denoting the beta function. Clearly, for $n=1$, Equation \eqref{enos} coincides with Equation \eqref{En}.

Entropy has recently been complemented by a dual measure, known as extropy  \cite{Lad}. 
The extropy associated with a random variable $X$ is defined as follows:
\begin{align}
J(X)=E\left[-\frac{1}{2}f(X)\right]=-\frac{1}{2}\int_{0}^{+\infty}f^2(x)dx.
\end{align}
Recently, \cite{Vaselabadi} introduced a new measure of uncertainty, termed varextropy, which serves as a complementary measure to varentropy and provides an alternative perspective on Shannon entropy. Based on the probability density function, varextropy is defined as follows:
\begin{align}\label{eqve}
\operatorname{VarJ}(X)&=\operatorname{Var}\left[-\frac{1}{2}f(X)\right]=\frac{1}{4}E(f^2(X))-J^2(X).
\end{align}   
The residual varextropy of $X$ at time $t$ is defined by
\begin{equation}
\operatorname{VarJ}(X; t)
=
\frac{1}{4}\mathbb{E}\left[f_{X_t}^2(X_t)\right]
-
J^2(X_t),
\end{equation}
where $X_t=(X-t\mid X>t)$ denotes the residual lifetime of $X$ at time $t$.
%

Hashempour et al. \cite{Hashem} had defined the measure of inaccuracy associated with two random variables $X$ and $Y$ as
\begin{align}\label{exsh}
J(X, Y)=J(X)+J(X|Y)=-\frac{1}{2}\int_{0}^{+\infty}f(x)g(x)dx,
\end{align}
where $J(X|Y)=\frac{1}{2}\int_{0}^{+\infty}f(x)[f(x)-g(x)]dx$ is discrimination measure of variable $X$ with respect to variable $Y$ based on extropy and inaccuracy between density functions $f(x)$ and $g(x)$.
It can be interpreted as the expected difference between the densities $f$ and $g$ with respect to $f$, providing a directional assessment of how $g$  deviates from $f$.

َ

Mohammadi et al. \cite{Moham} introduced two extropy-based inaccuracy measures associated with the ith order statistic $X_{i:n}$. 
The first of these measures is 
defined based on the distribution of $X_{i:n}$ as follows:
\begin{align}
J(X_{i:n})=
-\frac{1}{2}\int_{0}^{+\infty}f_{i:n}^2(x)\,dx=
-\frac{B(2i-1,\,2(n-i)+1)}
{2B^2(i,n-i+1)}E\!\left[f\!\left(F^{-1}(\mathscr{U}_2)\right)\right],
\end{align}
where $\mathscr{U}_2\sim Beta(2i-1, 2(n-i)+1)$.

In addition, they proposed an extropy-based inaccuracy measure between the ith order statistic $X_{i:n}$ and its  parent random variable $X$, given by
\begin{align}\label{eqJM}
J(X_{i:n},X)
&=-\frac{1}{2}\int_{0}^{+\infty}f_{i:n}(x)f(x)\,dx\nonumber\\&=-\frac{1}{2}\int_{0}^{+\infty}
\frac{F^{\,i-1}(x)\bar{F}^{\,n-i}(x)}{B(i,n-i+1)}f^2(x)\,dx\nonumber
\\&=-\frac{1}{2}E[f(F^{-1}(\mathscr{U}_1))],
\end{align}
where $\mathscr{U}_1\sim Beta(i, n-i+1)$.

Motivated by these two measures, the extropy-based discrimination measure between $X_{i:n}$ and its parent random variable  $X$ is given by
\begin{align}
J(X_{i:n}\!\mid X)
=\frac{B(2i-1,\,2(n-i)+1)}{2B^2(i,n-i+1)}E\!\left[f\left(F^{-1}(\mathscr{U}_2)\right)\right]-\frac{1}{2}E\!\left[f\left(F^{-1}(\mathscr{U}_1)\right)\right].
\end{align}
 The remainder of the paper is organized as follows: Section \ref{Section2}  is devoted to the introduction and investigation of a new variance-based inaccuracy measure for quantifying the variability of the inaccuracy function between an order statistic and its parent random variable $X$, based on the concept of varextropy. In particular, an explicit expression is obtained for the exponential distribution, for which a symmetry property with respect to the order statistics is established. For the Weibull distribution, since a closed-form expression is not available, an upper bound for the proposed measure is derived. Further bounds are obtained based on Chernoff’s inequality, and a transformation property under strictly increasing differentiable transformations is established. In Section \ref{Section3}, we develop a dynamic residual version of the proposed inaccuracy measure and establish characterization results for the underlying distribution. Bounds for the dynamic residual inaccuracy measures based on extropy and varextropy are also derived. In Section \ref{Section4}, a nonparametric estimator of the proposed measure is developed. In Section \ref{Section5}, the estimator is applied to a real dataset to illustrate its behavior in practical applications. 
\section{Varjinaccuracy for order statistics}\label{Section2}
For two non-negative continuous random variables $X$ and $Y$, with corresponding probability density functions $f$ and $g$ defined on the same support, \cite{Goodarzi} introduced the varjinaccuracy measure as a variance-based extension of the inaccuracy measure in terms of varextropy. It is defined as
\begin{align}\label{varj2}
\operatorname{VarJ}(X, Y)=\operatorname{Var}_f\left[-\frac{1}{2}g(X)\right]=\frac{1}{4}\int_{0}^{+\infty}g^2(x)f(x)dx-[J(X, Y)]^2.
\end{align}
In this section, we extend this measure to order statistics. 
Let $X_1, X_2, \ldots, X_n$ be a random sample from an absolutely continuous distribution with probability density function $f(x)$ and  distribution function $F(x)$, and let $X_{i:n}$ denote  the ith order statistic with probability density function $f_{i:n}$, then 
\begin{align}
\operatorname{VarJ}(X_{i:n}, X)&=\nonumber\operatorname{Var}_{f_{i:n}}
\!\left(-\frac{1}{2}f(X)\right)\\&=\frac{1}{4}
\left[
\int_{0}^{+\infty}f_{i:n}(x)f^{2}(x)\,dx
-
\left(
\int_{0}^{+\infty}f_{i:n}(x)f(x)\,dx
\right)^{2}
\right]
\nonumber\\&=\frac{1}{4}\left\{
E\left[f^2(F^{-1}(\mathscr{U}_1))\right]-E^2\left[f(F^{-1}(\mathscr{U}_1))\right]
\right\}.\label{e2.1}
\end{align} 
\begin{example}\label{exam2.1}
Let $X_1, X_2, \ldots, X_n$ be independent and identically distributed exponential random variables with probability density function 
\(f(x)=\lambda e^{-\lambda x}, \, x>0,\;\lambda>0,\)
and cumulative distribution function $F(x)=1-e^{-\lambda x}.$ Since
\[F^{-1}(u)=-\frac{1}{\lambda}\log(1-u),\]
we have
\[
f\!\left(F^{-1}(u)\right)=\lambda(1-u), \qquad 0<u<1.
\]
Also, 
\[
E(1-\mathscr{U}_1)=\frac{n-i+1}{n+1},\qquad\text{and}\qquad E\!\left[(1-\mathscr{U}_1)^2\right]
=
\frac{(n-i+1)(n-i+2)}
{(n+1)(n+2)},\]
and
hence,  by Equation \eqref{e2.1}, we obtain
\begin{align*}
\operatorname{VarJ}(X_{i:n},X)
&=
\frac{\lambda^{2}}{4}
\left[
\frac{(n-i+1)(n-i+2)}
{(n+1)(n+2)}
-\left(\frac{n-i+1}{n+1}\right)^2
\right]\\&=\frac{\lambda^{2}}{4}
\,\frac{i(n-i+1)}{(n+1)^2(n+2)}.
\end{align*}
Moreover,
\begin{align}
\operatorname{VarJ}(X_{i:n},X)
=\operatorname{VarJ}(X_{n-i+1:n},X),
\end{align}
which shows that the proposed measure is symmetric with respect to $i$.
In particular,
\[\operatorname{VarJ}(X_{1:n},X)=\operatorname{VarJ}(X_{n:n},X)=\frac{\lambda^{2}}{4}\,\frac{n}{(n+1)^2(n+2)}.\]
\end{example}
\begin{example}\label{exam2.2}
Let $X_1, X_2, \ldots, X_n$ be independent and identically distributed random variables following the Weibull distribution with probability density function
\[f(x)=\frac{a}{b}\left(\frac{x}{b}\right)^{a-1}\exp\!\left[-\left(\frac{x}{b}\right)^a\right],\]
and cumulative distribution function \(F(x)=1-\exp\!\left[-\left(\frac{x}{b}\right)^a\right],\)
where $x>0$, $a>0$, and $b>0$. 
Since
\[F^{-1}(u)=b\left[-\log(1-u)\right]^{1/a},\]
we have
\[f(F^{-1}(u))=\frac{a}{b}(1-u)\left[-\log(1-u)\right]^{\frac{a-1}{a}},\qquad 0<u<1.\]

Hence, by Equation \eqref{e2.1}
\[\begin{aligned}
\operatorname{VarJ}(X_{i:n},X)&=\frac{1}{4}\left[E\!\left(f^{2}(F^{-1}(\mathcal{U}_1))\right)-
\left(E\!\left(f(F^{-1}(\mathcal{U}_1))\right)\right)^2\right] \\
&=\frac{a^{2}}{4b^{2}}\Bigg[E\!\left((1-\mathcal{U}_1)^2
\left[-\log(1-\mathcal{U}_1)\right]^{\frac{2(a-1)}{a}}
\right) \\&\qquad\qquad\qquad-\left(E\!\left((1-\mathcal{U}_1)
\left[-\log(1-\mathcal{U}_1)\right]^{\frac{a-1}{a}}\right)\right)^2\Bigg],\end{aligned}\]
where $\mathcal{U}_1\sim\mathrm{Beta}(i, n-i+1)$.
\end{example}
The following theorem establishes a characterization of the exponential distribution based on the behavior of the density function under the quantile transformation.
\begin{theorem}
Let $X$ be a nonnegative absolutely continuous random variable with cumulative distribution function $F$ and 
probability density function $f$ and with support $S=[0, \infty)$. Then $f(F^{-1}(u))$
is an affine function of $u$, that is, 
\[
f(F^{-1}(u))=\alpha+\beta u,\qquad 0<u<1,
\]
for some constants $\alpha$ and $\beta$, if and only if $X$ follows an exponential distribution.
\end{theorem}
\begin{proof}
Suppose that $f(F^{-1}(u))=\alpha+\beta u,$ for $0\!<\!u<\!1$.
Setting $u=F(x)$ yields $f(x)=\alpha+\beta F(x).$
Since $f(x)=F'(x)$, we obtain the differential equation $F'(x)-\beta F(x)=\alpha.$
The general solution is 
\(F(x)=-\frac{\alpha}{\beta}+Ce^{\beta x},\)
where $C$ is a constant. Since $F(0)=0$, we obtain $C=\frac{\alpha}{\beta},$
and therefore
\(F(x)=\frac{\alpha}{\beta}\left(e^{\beta x}-1\right).\)
On the other hand, 
$\lim_{x\to\infty}F(x)=1,$ hence $\beta<0$. Writing $\beta=-\lambda$ with $\lambda>0$, we obtain
$F(x)=\frac{\alpha}{\lambda}\left(1-e^{-\lambda x}\right).$
Again, using $\lim_{x\to\infty}F(x)=1,$
it follows that $\alpha=\lambda$. Consequently, \(F(x)=1-e^{-\lambda x},\)
which is the cumulative distribution function of an exponential distribution.

Conversely, if $X$ follows an exponential distribution with parameter $\lambda$, then as shown in Example \ref{exam2.1},
\(f(F^{-1}(u))=\lambda-\lambda u,\)
which is an affine function of $u$.
This completes the proof.
\end{proof}
As illustrated by Example \ref{exam2.2}, the varjinaccuracy of the order statistic $X_{i:n}$ under the Weibull distribution admits an explicit representation; however, a closed-form expression is generally not available. This observation motivates the derivation of an upper bound for $\operatorname{VarJ}(X_{i:n},X)$, which is established in the following theorem.
\begin{theorem}
Suppose that $X$ follows a Weibull distribution with probability density
function
\[f(x)=\frac{a}{b}\left(\frac{x}{b}\right)^{a-1}
\exp\!\left[-\left(\frac{x}{b}\right)^a\right],\qquad x>0.\]
Then, the Chernoff upper bound for $\operatorname{VarJ}(X_{i:n},X)$ is
\[\operatorname{VarJ}(X_{i:n},X)\le\frac{an!}{4b^3(i-1)!(n-i)!}\sum_{k=0}^{i-1}
(-1)^k\binom{i-1}{k}\left[(a-1)^2A_0-2a(a-1)A_1+a^2A_2\right],\]
where, for $r=0, 1, 2$,
\[\begin{aligned}A_r=&\,\frac{\mu_{i:n}}{m_k}\Gamma\!\left(2-\frac{3}{a}+r\right)
\left[2^{-\left(2-\frac{3}{a}+r\right)}
-(m_k+2)^{-\left(2-\frac{3}{a}+r\right)}\right]
\\&-\frac{b\,\Gamma\!\left(3+r-\frac{2}{a}\right)}
{\left(1+\frac1a\right)(m_k+2)^{\,3+r-\frac{2}{a}}}\,{}_2F_1\!\left(1,\,3+r-\frac{2}{a};2+\frac1a;
\frac{m_k}{m_k+2}\right),
\end{aligned}\]
with
\[m_k=n-i+1+k,\]
and
\[\mu_{i:n}=b\Gamma\!\left(1+\frac1a\right)\frac{n!}{(i-1)!(n-i)!}\sum_{k=0}^{i-1}
(-1)^k\binom{i-1}{k}m_k^{-\left(1+\frac1a\right)}.\]
\end{theorem}
\begin{proof}
Applying the Chernoff inequality given in Equation (12) of \cite{Goodd} to the proposed inaccuracy measure based on varextropy yields
\begin{align}\label{eqine}
\operatorname{VarJ}(X_{i:n},X)\le\frac14\int_{0}^{+\infty}\left(f'(x)\right)^2
\left(\int_{0}^{x}(\mu_{i:n}-t)f_{i:n}(t)\,dt\right)dx.
\end{align}

Taking the logarithmic derivative of $f(x)$ and using Equation \eqref{eq1.4}, together with the binomial expansion,
\[\left(1-e^{-(x/b)^a}\right)^{i-1}=\sum_{k=0}^{i-1}(-1)^k\binom{i-1}{k}
e^{-k(x/b)^a},\]
we obtain 
\begin{align}\label{deriv}
f'(x)=\frac{a}{b^{2}}\left(\frac{x}{b}\right)^{a-2}
\left[(a-1)-a\left(\frac{x}{b}\right)^{a}\right]
\exp\!\left[-\left(\frac{x}{b}\right)^{a}\right],
\end{align}
and
\begin{align}
f_{i:n}(t)=\frac{n!}{(i-1)!(n-i)!}\left(\frac{a}{b}\right)
\left(\frac{t}{b}\right)^{a-1}
\sum_{k=0}^{i-1}{(-1)}^k{i-1\choose k}e^{-(n-i+1+k)({t}/{b})^a}.
\end{align}
 Thus, upon setting $m_k=n-i+1+k$, we obtain  
 \begin{align}\label{eqcher}
\int_{0}^x(\mu_{i: n}-t)f_{i: n}(t)dt&=\frac{n!}{(i-1)!(n-i)!}\sum_{k=0}^{i-1}{(-1)}^k{i-1\choose k}\left\{\mu_{i: n}\frac{1-e^{-m_k(x/b)^a}}{m_k}
\right.\nonumber\\&\left.-\int_{0}^{x}t\left(\frac{a}{b}\right){\left(\frac{t}{b}\right)}^{a-1}e^{-m_k(t/b)^a}dt\right\}.
 \end{align}
 Making the change of variable $u=m_k(t/b)^a$ and using the definition of the incomplete gamma function, 
 $\gamma(s, mz)=m^s\int_{0}^zv^{s-1}e^{-mv}dv$,
  we have
 \begin{align}\label{eqJ1}
 J_1(m_k, x)&=\int_{0}^xt\left(\frac{a}{b}\right){\left(\frac{t}{a}\right)}^{a-1}e^{-m_k(t/b)^a}dt=\frac{b}{m_k^{1+\frac{1}{a}}}\int_{0}^{m_k(x/b)^a}u^{1/a}e^{-u}du\nonumber\\&=\frac{b}{m_k^{1+\frac{1}{a}}}\gamma\left(1+\frac{1}{a}, m_k\left(\frac{x}{b}\right)^a\right).
 \end{align}
 Now 
 \begin{align}
 \mu_{i:n}&=\frac{n!}{(i-1)!(n-i)!}\sum_{k=0}^{i-1}{(-1)}^k{i-1\choose k}\int_{0}^{+\infty}x\left(\frac{a}{b}\right){\left(\frac{x}{b}\right)}^{a-1}e^{-m_k(x/b)^a}dx\nonumber\\
 &=\frac{n!}{(i-1)!(n-i)!}\sum_{k=0}^{i-1}{(-1)}^k{i-1\choose k}J_1(m_k, \infty)
 \nonumber\\&=b\Gamma\!\left(1+\frac1a\right)
\frac{n!}{(i-1)!(n-i)!}\sum_{k=0}^{i-1}(-1)^k\binom{i-1}{k}
m_k^{-\left(1+\frac1a\right)}.
 \end{align}
Substituting \eqref{eqJ1} into \eqref{eqcher} and 
then substituting the resulting expression into the Chernoff inequality \eqref{eqine}, we obtain 
\begin{equation}
\operatorname{VarJ}(X_{i:n},X)\le\frac{n!}{4(i-1)!(n-i)!}\sum_{k=0}^{i-1}(-1)^k\binom{i-1}{k}I_k.
\end{equation}
where 
\begin{align}
I_k=\int_{0}^{+\infty}(f^{'}(x))^2\left[\frac{\mu_{i:n}}{m_k}(1-e^{-m_k(x/b)^a})-\frac{b}{m_k^{1+\frac{1}{a}}}\gamma\left(1+\frac{1}{a},m_k\left(\frac{x}{b}\right)^a\right)\right]dx.
\end{align}
Expanding the quadratic polynomial \(\left[(a-1)-a\left(\frac{x}{b}\right)^a\right]^2\) and making the change of variable 
$u=\left(\frac{x}{b}\right)^a,$
we can decompose the integral $I_k$ as


\[I_k=\frac{a}{b^3}\left[(a-1)^2A_0-2a(a-1)A_1+a^2A_2\right],\]
where,
\begin{align*}
A_r&=\int_{0}^{+\infty}u^{1-\frac{3}{a}+r}e^{-2u}\left[\frac{\mu_{i:n}}{m_k}(1-e^{-m_ku})-\frac{b}{m_k^{1+\frac{1}{a}}}
\gamma\left(1+\frac{1}{a},m_ku\right)^a\right]du\\&=
\frac{\mu_{i:n}}{m_k}\Gamma\!\left(2-\frac3a+r\right)
\left[2^{-\left(2-\frac3a+r\right)}-(m_k+2)^{-\left(2-\frac3a+r\right)}\right]\\&
-b\int_{0}^{+\infty}
u^{1-\frac{3}{a}+r}e^{-2u}\left(\int_{0}^{u}v^{{1}/{a}}e^{-mv}dv\right)du\\&=
\frac{\mu_{i:n}}{m_k}\Gamma\!\left(2-\frac3a+r\right)
\left[2^{-\left(2-\frac3a+r\right)}-(m_k+2)^{-\left(2-\frac3a+r\right)}\right]\\&
-b\int_{0}^{+\infty}v^{{1}/{a}}e^{-mv}\left(\int_{v}^{+\infty}u^{1-\frac{3}{a}+r}e^{-2u}du\right)dv
\\&=
\frac{\mu_{i:n}}{m_k}\Gamma\!\left(2-\frac3a+r\right)
\left[2^{-\left(2-\frac3a+r\right)}-(m_k+2)^{-\left(2-\frac3a+r\right)}\right]\\&
-\frac{b}{2^{2-\frac{3}{a}+r}}\int_{0}^{+\infty}v^{{1}/{a}}e^{-mv}\Gamma\left(2-\frac{3}{a}+r, 2v\right)dv.
\end{align*}

Using the identity given in  \cite{Gradshteyn} p. 657,
\[\int_{0}^{+\infty}x^{\mu-1}e^{-\beta x}\Gamma(\nu,\alpha x)\,dx
=\frac{\alpha^\nu\Gamma(\mu+\nu)}{\mu(\alpha+\beta)^{\mu+\nu}}\,{}_2F_1\!\left(1,\mu+\nu;\mu+1;
\frac{\beta}{\alpha+\beta}\right).\]
Setting
\[\mu=1+\frac1a,\qquad\nu=2-\frac3a+r,\qquad\alpha=2,\qquad\beta=m_k,\]
in the above identity yields
\[\begin{aligned}A_r=&\,\frac{\mu_{i:n}}{m_k}\Gamma\!\left(2-\frac3a+r\right)
\left[2^{-\left(2-\frac3a+r\right)}-(m_k+2)^{-\left(2-\frac3a+r\right)}\right]\\&-
\frac{b\,\Gamma\!\left(3+r-\frac2a\right)}{\left(1+\frac1a\right)
(m_k+2)^{3+r-\frac2a}}\,{}_2F_1\!\left(1,\,3+r-\frac2a;2+\frac1a;\frac{m_k}{m_k+2}\right).\end{aligned}\]

Substituting the obtained expression for $A_r$ into $I_k$, and subsequently  into the above Chernoff upper bound, yields the desired result.
\end{proof}

\begin{theorem}\label{2th}
Let $X$ be a non-negative continuous random variable with cumulative distribution function $F(x)$ and probability
density function $f(x)$. Define the random variable $Y=\xi(X)$,
where $\xi(\cdot)$ is a strictly increasing and differentiable function with derivative $\xi^{'}(x)$. Furthermore, let $G(y)$ and $g(y)$ denote the cumulative distribution
function and the probability density function of \(Y\), respectively. Also, let $X_{i:n}$ and $Y_{i:n}$ 
represent the $i{th}$ order statistics corresponding to $X$ and $Y$, with associated probability density 
functions \(f_{i:n}(x)\) and \(g_{i:n}(y)\), respectively. Then
\begin{align}
\operatorname{VarJ}(Y_{i:n},Y)=\frac{1}{4}\left[E_{f_{i:n}}\left(\frac{f^{2}(X)}{\left(\xi'(X)\right)^{2}}\right)
-\left(E_{f_{i:n}}\left(\frac{f(X)}{\xi'(X)}\right)
\right)^{2}\right].
\end{align}
\end{theorem}
\begin{proof}
Since $Y=\xi(X)$, where $\xi$ is a strictly increasing differentiable function, we have 
$G(y)=F(\xi^{-1}(y))$ and 
differentiating both sides with respect to $y$ and using the identity
$(\xi^{-1})'(y)=\frac{1}{\xi'(\xi^{-1}(y))},$ it follows that
\(g(y)=\frac{f(\xi^{-1}(y))}{\xi^{'}(\xi^{-1}(y))}.\)
Also $Y_{i:n}=\xi(X_{i:n})$, thus
\begin{align}
\operatorname{VarJ}(Y_{i:n}, Y)&=
\operatorname{Var}_{g_{i: n}}\left[-\frac{1}{2}g(Y)\right]\nonumber\\&=\frac{1}{4}\left[\int_{0}^{+\infty}\frac{1}{B(i, n-i+1)}G^{i-1}(y)
{(1-G(y))}^{n-i}g^3(y)dy\right.\nonumber\\&\left.-{\left(\int_{0}^{+\infty}\frac{1}{B(i, n-i+1)}G^{i-1}(y)
{(1-G(y))}^{n-i}g^2(y)dy\right)}^2\right]\nonumber
\\&=
\frac{1}{4}\left[\int_{0}^{+\infty}\frac{1}{B(i, n-i+1)}{\left[F(\xi^{-1}(y))\right]}^{i-1}
{\left[1-F(\xi^{-1}(y))\right]}^{n-i}\left[\frac{f(\xi^{-1}(y))}{\xi^{'}(\xi^{-1}(y))}\right]^3dy\right.\nonumber\\&\left.-{\left(\int_{0}^{+\infty}
\frac{1}{B(i, n-i+1)}
{\left[F(\xi^{-1}(y))\right]}^{i-1}
{\left[1-F(\xi^{-1}(y))\right]}^{n-i}
\left[\frac{f(\xi^{-1}(y))}{\xi^{'}(\xi^{-1}(y))}\right]^2dy\right)}^2\right].\nonumber
\end{align}
Now, substituting $x=\xi^{-1}(y)$, such that $dy=\xi^{'}(x)dx$, we have
\begin{align}
\operatorname{VarJ}(Y_{i:n}, Y)&=
\frac{1}{4}\left[\int_{0}^{+\infty}\frac{1}{B(i, n-i+1)}{\left[F(x)\right]}^{i-1}
{\left[1-F(x)\right]}^{n-i}\left[\frac{f(x)}{\xi^{'}(x)}\right]^2f(x)dx\right.\nonumber\\&\left.-{\left(\int_{0}^{+\infty}
\frac{1}{B(i, n-i+1)}
{\left[F(x)\right]}^{i-1}
{\left[1-F(x)\right]}^{n-i}
\left[\frac{f(x)}{\xi^{'}(x)}\right]f(x)dx\right)}^2\right]\nonumber
\\
&=\frac{1}{4}\left[\int_{0}^{+\infty}\left[\frac{f(x)}{\xi^{'}(x)}\right]^2f_{i:n}(x)dx\right.
\left.-{\left(\int_{0}^{+\infty}
\left[\frac{f(x)}{\xi^{'}(x)}\right]f_{i:n}(x)dx\right)}^2\right],\nonumber
\end{align}
which completes the proof. 
\end{proof}
The next corollary shows that the proposed varjinaccuracy measure is invariant under location transformations but not under scale transformations. 
%
%
\begin{corollary}
Let $X$ be a nonnegative continuous random variable with cumulative distribution function $F(x)$ and probability density function $f(x)$. Suppose that
\[Y=aX+b,\]
where $a>0$ and $b\in\mathbb{R}$. Let $X_{i:n}$ and $Y_{i:n}$ denote the corresponding $i$th order statistics. Then,
\[\operatorname{VarJ}(Y_{i:n},Y)
=\frac{1}{a^{2}}
\operatorname{VarJ}(X_{i:n},X).\]
\end{corollary}
\begin{proof}
Since $\xi(x)=ax+b$, it follows that $\xi^{'}\!(x)=a$. Therefore, by Theorem \ref{2th},
\begin{align*}
\operatorname{VarJ}(Y_{i:n},Y)
&=
\frac{1}{4}
\left\{
E_{f_{i:n}}\!\left[\left(\frac{f(X)}{a}\right)^2\right]-\left(E_{f_{i:n}}\!\left[\frac{f(X)}{a}\right]\right)^2\right\}
\\&=\frac{1}{a^{2}}
\operatorname{VarJ}(X_{i:n},X).
\end{align*}
This completes the proof.
\end{proof}
\begin{remark}
Let $X$ be a nonnegative absolutely continuous random variable, and let
$X_{i:n}$ denote the $i$th order statistic from a random sample of size $n$.
Assume that $X_{i:n}$ belongs to the Integrated Pearson family with quadratic
function $q(\cdot)$. Then, applying inequality (3.5) of \cite{Goodarzi} (see Remark 3.6 of therein) with 
$g(x)=f(x)$, we obtain
\begin{align}\label{e2.5}
\operatorname{VarJ}(X_{i:n},X)
&=\frac{1}{4}\operatorname{Var}_{f_{i:n}}\!\left(f(X)\right)\nonumber\\
&\ge\frac{1}{4}\sum_{k=1}^{n}\frac{\left(E_{f_{i:n}}\left[q^{k}(X)\,f^{(k)}(X)\right]\right)^{2}}
{k!\,E_{f_{i:n}}\left[q^{k}(X)\right]\displaystyle\prod_{j=k-1}^{2k-2}(1-j\delta)},
\end{align}
where $f^{(k)}$ denotes the $k$th derivative of the parent density function $f(\cdot)$, whereas $q(x)$ and $\delta$ correspond to  density $f_{i:n}(\cdot)$.
\end{remark}
\begin{example}
For illustration, let the parent random variable \(X\) follow a Beta
distribution with parameters $a$ and $1$, with probability density
function
\[
f(x)=a x^{a-1}, \qquad 0<x<1,\quad a>0.
\]
It follows that $X_{n:n}$  follows a beta distribution with parameters $an$ and $1$.
For the density \(f_{n:n}(\cdot)\), the corresponding function $q(\cdot)$ and coefficient $\delta$ are given by
\[q(x)=\frac{x(1-x)}{an+1}, \qquad\delta=-\frac{1}{an+1}.\]
Moreover,
\[f^{(k)}(x)=a\prod_{r=1}^{k}(a-r)x^{a-k-1},\]
which gives
\[E_{f_{n:n}}\!\left[q^{k}(X)f^{(k)}(X)\right]=\frac{a^2n\displaystyle\prod_{r=1}^{k}(a-r)}
{(an+1)^k}B\!\left(a(n+1)-1,k+1\right)\]
and
\[E_{f_{n:n}}\!\left[q^{k}(X)\right]=\frac{an}{(an+1)^k}B(an+k,k+1).\]
Substituting these expressions into the lower bound \eqref{e2.5} and using
\(B(x,k+1)=k!\Gamma(x)/\Gamma(x+k+1)\), we obtain
\[
\begin{aligned}
\operatorname{VarJ}(X_{n:n},X)
\geq{}&
\frac{a^3n}{4}
\sum_{k=1}^{n}
\frac{
\Gamma^2\!\left(a(n+1)-1\right)
\Gamma(an+2k+1)
\left[\displaystyle\prod_{r=1}^{k}(a-r)\right]^2
}{
\Gamma^2\!\left(a(n+1)+k\right)
\Gamma(an+k)
\displaystyle\prod_{j=k-1}^{2k-2}(an+1+j)
}.
\end{aligned}
\]
\end{example}
\begin{theorem}
For a nonnegative continuous random variable $X$, let $X_{1:n}$, $X_{2:n},$ $\ldots,$ $X_{n:n}$ denote the order statistics from a random sample of size $n$. Then, the average value of the proposed inaccuracy the $i$th order statistics and the
parent random variable based on varextropy is bounded above by the varextropy of the parent random variable $X$, that is,
\[\frac{1}{n}\sum_{i=1}^{n}\operatorname{VarJ}(X_{i:n},X)\leq\operatorname{VarJ}(X).\]
\end{theorem}
\begin{proof}
Taking the average over $i=1, \ldots, n$ and using the identity
\[\frac{1}{n}\sum_{i=1}^{n}f_{i:n}(x)=f(x),\]
we obtain
\begin{align*}
\frac{1}{n}\sum_{i=1}^{n}\operatorname{VarJ}(X_{i:n},X)
=\frac14\left[
\int_{0}^{+\infty}f^{3}(x)\,dx-\frac1n\sum_{i=1}^{n}
\left(\int_{0}^{+\infty}f(x)f_{i:n}(x)\,dx\right)^2\right].
\end{align*}
Let \(a_i=\int_{0}^{+\infty}f(x)f_{i:n}(x)\,dx,\)
so that
\(\frac1n\sum_{i=1}^{n}a_i=\int_{0}^{+\infty}f^{2}(x)\,dx.\)
Since $x^2$ is a convex function, Jensen's inequality yields
\[\frac1n\sum_{i=1}^{n}a_i^2\ge\left(\frac1n\sum_{i=1}^{n}a_i\right)^2=
\left(\int_{0}^{+\infty}f^{2}(x)\,dx\right)^2.\]
Therefore,
\[\frac1n\sum_{i=1}^{n}\operatorname{VarJ}(X_{i:n},X)\le
\frac14\left[\int_{0}^{+\infty}f^{3}(x)\,dx-\left(\int_{0}^{+\infty}f^{2}(x)\,dx\right)^2\right]=\operatorname{VarJ}(X),\]
which completes the proof.
\end{proof}
\section{Dynamic Residual Inaccuracy Measure for Order Statistics Based on varextropy}\label{Section3}
In survival analysis and life-testing, the uncertainty associated with a system may depend on its 
survival up to a given time \(t\). This motivates the consideration of dynamic measures that account for the residual lifetime beyond \(t\). In this section, we extend the varjinaccuracy measure to the residual lifetime of order statistics and define its dynamic version based on varextropy as follows.
\begin{definition}
The dynamic residual inaccuracy measure for order statistics based on varextropy is defined as
\begin{equation}\label{v3.1}
\operatorname{VarJ}(X_{i:n},X;t)
=\frac{1}{4}\left\{
\int_{t}^{+\infty}\left(\frac{f(x)}{\bar{F}(t)}
\right)^2\frac{f_{i:n}(x)}{\bar{F}_{i:n}(t)}\,dx-
\left(\int_{t}^{+\infty}\frac{f(x)}{\bar{F}(t)}
\frac{f_{i:n}(x)}{\bar{F}_{i:n}(t)}\,dx\right)^2\right\},
\end{equation}
where
\[\bar{F}_{i:n}(t)=1-F_{i:n}(t),\]
denotes the survival function of the $i$th order statistic, that given by
\begin{align}
\bar{F}_{i:n}(t)=\frac{\int_{F(t)}^1u^{i-1}(1-u)^{n-i}du}{B({i, n-i+1})}=
\frac{\bar B_{F(t)}(i, n-i+1)}{B({i, n-i+1})}.
\end{align}
\end{definition}
\begin{example}
Assume that $X$ follows the exponential distribution with probability density function given in Example \ref{exam2.2}. Then, after some straightforward calculations, it can be shown that
\begin{equation}
\operatorname{VarJ}(X_{i:n},X;t)
=
\frac{\lambda^{2}}{4}
\frac{i(n-i+1)}
{(n+1)^{2}(n+2)},
\qquad i=1,2,\ldots,n.
\end{equation}

As can be seen,
\[\operatorname{VarJ}(X_{i:n},X;t)=\operatorname{VarJ}(X_{i:n},X),\]
which shows that the proposed dynamic residual varextropy-based inaccuracy measure is independent of the truncation time $t$.
\end{example}
\begin{example}\label{exam3.3}
Assume that the random variable $X$ follows the beta distribution with probability density function
\[
f(x)=a(1-x)^{a-1}, \qquad a>1,\; 0<x<1.
\]
Next, straightforward algebraic manipulation gives 
\begin{equation}\label{e3.4}
\operatorname{VarJ}(X_{1:n},X;t)
=
\frac{na^{3}}{4}
\left[
\frac{1}{an+2a-2}-
\frac{na}{(an+a-1)^{2}}
\right]
(1-t)^{-2}.
\end{equation}
It follows that the proposed dynamic residual varextropy-based inaccuracy measure explicitly depends on the truncation time $t$. Consequently,
\[
\operatorname{VarJ}(X_{1:n},X;t)
\neq
\operatorname{VarJ}(X_{1:n},X).
\]
\begin{figure}[ht]
\centering
\begin{minipage}{0.48\textwidth}
    \centering
    \includegraphics[width=\textwidth,height=6cm]{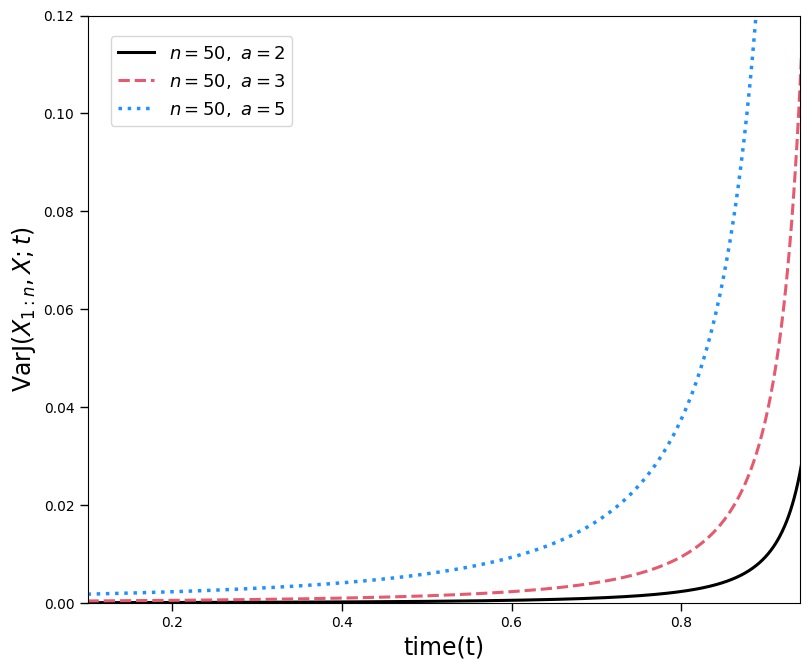}
\end{minipage}
\hfill
\begin{minipage}{0.48\textwidth}
    \centering
    \includegraphics[width=\textwidth,height=6cm]{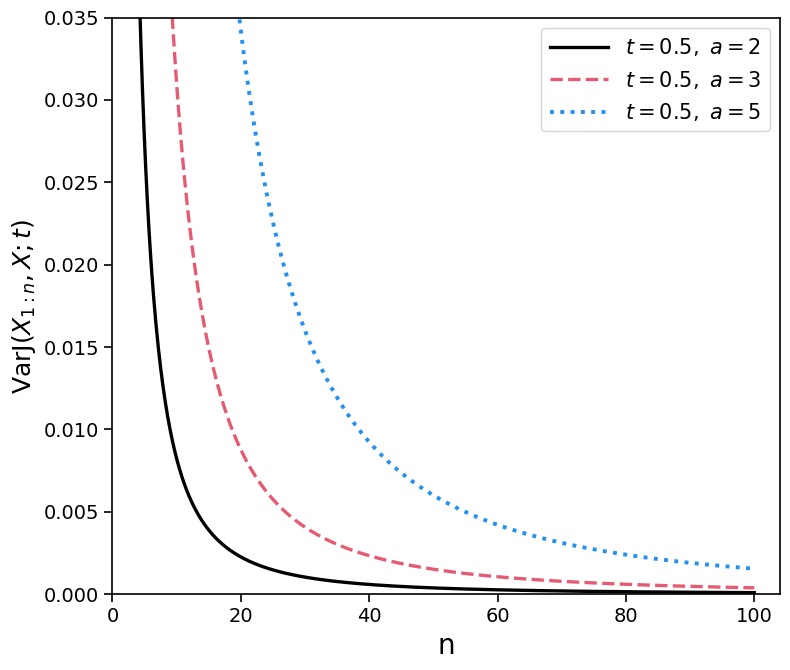}
\end{minipage}

\vspace{-0.2cm}
\caption{\footnotesize{Plots illustrating the behavior of $\mathrm{VarJ}(X_{1:n},X;t)$ for different values of the parameter $a$ as a function of time $t$ (left panel) and sample size $n$ (right panel) in Example \ref{exam3.3}.}}
\label{fig2}
\end{figure}

Figure \ref{fig2} illustrates the behaviour of the proposed dynamic residual varextropy-based inaccuracy measure for different values of the parameter $a$. The left panel of Figure \ref{fig2} shows that $\mathrm{VarJ}(X_{1:n},X;t)$ increases with increasing time $t$ for a fixed sample size $n$, with 
larger values of $a$ resulting in higher $\mathrm{VarJ}(X_{1:n},X;t)$ values. The right panel of Figure  \ref{fig2} shows that $\mathrm{VarJ}(X_{1:n},X;t)$ 
decreases as the sample size $n$ increases for a fixed value of $t$, and tends to zero as $n$ becomes large.

We also establish analytically that Equation \eqref{e3.4} is decreasing with respect to n. 
\begin{proof}
Let
\[
V_n=
\frac{a^3}{4(1-t)^2}
\left[
\frac{n}{an+2a-2}
-
\frac{an^2}{(an+a-1)^2}
\right],
\qquad a>1,
\]
denote $\operatorname{VarJ}(X_{1:n},X;t)$. To prove that $V_n$ is decreasing in $n$, it suffices to show that
\[
V_{n+1}\le V_n.
\]

After straightforward algebraic manipulations, this inequality is equivalent to
\[
\frac{(an+2a-2)(an+a-1)^2}
{(an+3a-2)(an+2a-1)^2}
\le
\frac{n}{n+1},
\]
or, equivalently,
\[
n(an+3a-2)(an+2a-1)^2
-
(n+1)(an+2a-2)(an+a-1)^2
\ge0.
\]

Expanding the left-hand side gives
\[P(n)=
2a^3n^3+7a^3n^2+5a^3n-4a^2n^2-4a^2n-2a^3+6a^2-6a+2.\]

Instead of analyzing $P(n)$ directly, we show that it is increasing with respect to $n$. Indeed,
\[P(n+1)-P(n)=2a^2(n+1)(3an+7a-4).\]
Since
\[3an+7a-4\ge 3a+7a-4=10a-4>0,\]
hence,
\[P(n+1)-P(n)>0,\]
which shows that $P(n)$ is strictly increasing in $n$. Consequently,
\[P(n)\ge P(1),\qquad n\ge1.\]

Moreover, \(P(1)=2(2a-1)(3a^2+a-1),\)
which is strictly positive for every $a>1$. Therefore,
\[P(n)\ge P(1)>0,\qquad n\ge1.\]

Thus,
\[n(an+3a-2)(an+2a-1)^2\ge (n+1)(an+2a-2)(an+a-1)^2,\]
which is equivalent to $V_{n+1}\le V_n.$ Hence, $\operatorname{VarJ}(X_{1:n},X;t)$ is a decreasing function of $n$.
\end{proof}
\end{example}
%
%
In following theorem, we aim to establish a relationship between the 
mean residual lifetime of a unit and the 
mean lifetime of a parallel system. The purpose of presenting 
this inequality is to obtain a lower bound for
$\operatorname{VarJ}(X_{i:n},X;t)$ based on the 
Chernoff bound of the dynamic varextropy of the 
underlying distribution.
\begin{proposition}
Let $X$ be a nonnegative absolutely continuous random variable with survival function $\overline F$, and assume that $\overline F(t)>0$. Let
\[m(t)=\frac{\displaystyle\int_t^{+\infty} \overline F(x)\,dx}{\overline F(t)}\]
denote the mean residual life function of $X$, and let
\[m_{n:n}(t)=\frac{\displaystyle\int_t^{+\infty}
\left[1-\left(1-\overline F(x)\right)^n\right]dx}
{1-\left(1-\overline F(t)\right)^n},\]
be the mean residual life function of the largest order statistic $X_{n:n}$. Then
\[m_{n:n}(t)\ge m(t),\]
for all $t$ such that $\overline F(t)>0$.
\end{proposition}
%
%
\begin{proof}
Define
\[
\phi(u)=1-(1-u)^n,\qquad 0<u\le1,
\]
and
\[
R(u)=\frac{\phi(u)}{u}
=\frac{1-(1-u)^n}{u}.
\]

\medskip
\noindent
 Differentiating $R(u)$ gives
\[R'(u)
=\frac{nu(1-u)^{n-1}-1+(1-u)^n}{u^2}.\]
Now let
\[h(u)=nu(1-u)^{n-1}-1+(1-u)^n.\]
Since $h(0)=0,$ and
\[h'(u)=-n(n-1)u(1-u)^{n-2}\le0,\]
it follows that
\[
h(u)\le0,
\qquad 0<u\le1.
\]
Consequently, $R'(u)\le0,$
and therefore $R$ is decreasing on $(0,1]$.
Since $\phi(u)=uR(u),$ we may write
\[
m_{n:n}(t)
=
\frac{\displaystyle\int_t^{+\infty}
R(\overline F(x))\,\overline F(x)\,dx}
{R(\overline F(t))\,\overline F(t)}.
\]

Now, for every $x\ge t$,
\(
\overline F(x)\le \overline F(t).
\)
Since $R$ is decreasing, \(R(\overline F(x))\ge R(\overline F(t)).\)
Therefore,
\[
\int_t^{+\infty}
R(\overline F(x))\,\overline F(x)\,dx
\ge
R(\overline F(t))
\int_t^{+\infty}
\overline F(x)\,dx.
\]

Finally, dividing both sides by the positive quantity
\(
R(\overline F(t))\,\overline F(t),
\)
we obtain
\[m_{n:n}(t)\ge \frac{\displaystyle\int_t^{+\infty}\overline F(x)\,dx}
{\overline F(t)}=m(t),\]
which completes the proof.
\end{proof}
%
%
\begin{theorem}
Let $X$ be a nonnegative absolutely continuous random variable, and let $X_{1:n}$ denote the smallest order statistic from a random sample of size $n$. Then
\[\operatorname{VarJ}(X_{1:n},X;t)\le
n\,\operatorname{UpperBound}\!\left(\operatorname{VarJ}(X;t)\right).\]
\end{theorem}
\begin{proof}
By Equations \eqref{v3.1} and (12) of \cite{Goodd}, we have
\[\operatorname{VarJ}(X_{1:n},X;t)\leq
\frac{1}{4}\int_{t}^{+\infty}\int_{t}^{x}\bigl(m_{1: n}(t)-s\bigr)\,f^{t}_{1:n}(s)\,
\left(\frac{f^{'}(x)}{\overline F(t)}\right)^2\,ds\,dx,\]
where
\[f^{t}_{1:n}(s)=n\left(\frac{\overline F(s)}
{\overline F(t)}\right)^{n-1}\frac{f(s)}{\overline F(t)}.\]
Since
\(\left(\frac{\overline F(s)}
{\overline F(t)}\right)^{n-1}
\le1,\)
it follows that
\[f^{t}_{1:n}(s)\le n\,\frac{f(s)}{\overline F(t)}=n\,f^{t}(s).\]

On the other hand, it is readily verified that $m_{1:n}(t)\leq m(t)$, therefore,
\begin{align}
\operatorname{VarJ}(X_{1:n},X;t)\nonumber
&\le \frac{n}{4}\int_{t}^{+\infty}
\int_{t}^{x}\bigl(m(t)-s\bigr)
\,f^{t}(s)\,
\left(\frac{f^{'}(x)}{\overline F(t)}\right)^2\,
ds\,dx\\&=
n\,\operatorname{UpperBound}\!\left(
\operatorname{VarJ}(X;t)
\right).
\end{align}
\end{proof}
\begin{theorem}
Let $X$ be a nonnegative absolutely continuous random variable, and let $X_{n:n}$ denote the largest order statistic from a random sample of size $n$. Then
\[\operatorname{VarJ}(X_{n:n},X;t)\ge\frac{n^2\sigma^2(t)}{\sigma^2_{n:n}(t)}
\left(\frac{F^{n-1}(t)}{\sum_{k=0}^{n-1}F^k(t)}\right)^2\,\operatorname{LowerBound}\!\left(\operatorname{VarJ}(X;t)\right).\]
\end{theorem}
\begin{proof}
By Equations \eqref{v3.1} and (3.3a) of \cite{Cacu}, we have
\[\operatorname{VarJ}(X_{n:n},X;t)\geq
\frac{1}{4\sigma^2_{n:n}(t)}\left(\int_{t}^{+\infty}\int_{t}^{x}\bigl(m_{n: n}(t)-s\bigr)\,f^{t}_{n:n}(s)\,
\left(\frac{f^{'}(x)}{\overline F(t)}\right)\,ds\,dx\right)^2,\]
where
\[f^{t}_{n:n}(s)=\frac{nF^{n-1}(s)f(s)}{1-F^n(t)},\quad s>t.\]
Since
$1-F^n(t)=\overline F(t)\sum_{k=0}^{n-1}F^{k}(t)$
it follows that
\[f^{t}_{n:n}(s)\ge \frac{nF^{n-1}(t)}{\sum_{k=0}^{n-1}F^{k}(t)}\,\frac{f(s)}{\overline F(t)}= \frac{nF^{n-1}(t)}{\sum_{k=0}^{n-1}F^{k}(t)}\,f^{t}(s).\]
On the other hand, $m_{n:n}(t)\geq m(t)$, therefore,
\begin{align}
\operatorname{VarJ}(X_{n:n},X;t)\nonumber
&\ge \frac{1}{4\sigma^2_{n:n}(t)}{\left(\int_{t}^{+\infty}
\int_{t}^{x}\bigl(m(t)-s\bigr)
\,\frac{nF^{n-1}(t)}{\sum_{k=0}^{n-1}F^{k}(t)}f^{t}(s)\,
\left(\frac{f^{'}(x)}{\overline F(t)}\right)\,
ds\,dx\right)}^2\\&=
\frac{\sigma^2(t)}{\sigma^2_{n:n}(t)}\left(\frac{nF^{n-1}(t)}{\sum_{k=0}^{n-1}F^{k}(t)}\right)^2\,\operatorname{LowerBound}\!\left(
\operatorname{VarJ}(X;t)
\right),
\end{align}
where $\sigma^2_{n: n}(t)=\frac{n}{1-F^n(t)}\int_{0}^{+\infty}(x-t)^2F^{n-1}(x)f(x)dx-m_{n:n}^2(t)$.
\end{proof}
%
%
\begin{theorem}\label{thm:BoundsJ}
Let $X$ be a nonnegative absolutely continuous random variable with
continuously differentiable density function $f$. Let
\[
m_{i:n}(t)
=
E\!\left(X_{i:n}-t\,\middle|\,X_{i:n}>t\right),
\]
denote the mean residual life (MRL) function of the $i$th order statistic.
Furthermore, let
\[
\eta(t)
=
-\frac{f'(t)}{f(t)},
\]
denote the \emph{eta function} associated with the density function $f$. Then the following assertions hold.
\begin{enumerate}
\item[(i)]
If $f$ is 
 concave on its support, then
\[J(X_{i:n},X;t)\ge\frac{r(t)}{2}\left(\eta(t)m_{i:n}(t)-1\right).\]

\item[(ii)]
If $f$ is 
convex on its support, then
\[J(X_{i:n}, X;t)\le\frac{r(t)}{2}\left(\eta(t)m_{i:n}(t)-1\right),\]
\end{enumerate}
where $r(t)=\frac{f(t)}{\overline F(t)}$ is hazard function of $X$.
\end{theorem}
\begin{proof}
Recall from \eqref{eqJM} that
\[J(X_{i:n},X;t)=
-\frac{1}{2\overline F(t)}
E\!\left[f(X_{i:n})
\,\middle|\,X_{i:n}>t\right].\]

Define
\[A(t)=\frac{1}{\overline F(t)}E\!\left[f(X_{i:n})\,\middle|\,X_{i:n}>t\right].\]

therefore
\[
r(t)-A(t)
=
\frac{
f(t)
-
E\!\left[f(X_{i:n})\,\middle|\,X_{i:n}>t\right]
}
{\overline F(t)}.
\]

Now let $x>t$. By the Mean Value Theorem,
\[
f(t)-f(x)
=
-f'(\xi)(x-t),
\qquad
\xi\in(t,x).
\]

Assume that $f$ is concave. Since \(f''(x)\le0,\)
the derivative $f'$ is decreasing. Hence,
\(
-f'(\xi)\ge -f'(t),
\)
which implies
\[
f(t)-f(x)
\ge
(-f'(t))(x-t).
\]

Taking conditional expectation with respect to the event
$\{X_{i:n}>t\}$,
\[
f(t)
-
E\!\left[f(X_{i:n})\,\middle|\,X_{i:n}>t\right]
\ge
(-f'(t))
m_{i:n}(t).
\]

Dividing both sides by $\overline F(t)$ yields
\[
r(t)-A(t)
\ge
\frac{-f'(t)}
{\overline F(t)}
m_{i:n}(t).
\]

Since
\(-f'(t)=\eta(t)f(t)=\eta(t)r(t)\overline F(t),\)
we obtain
\[r(t)-A(t)\ge\eta(t)r(t)m_{i:n}(t).\]

Finally, using the identity $A(t)=-2J(X_{i:n},X;t),$
we have
\[
r(t)+2J(X_{i:n},X;t)
\ge
\eta(t)r(t)m_{i:n}(t),
\]
or equivalently,
\[
J(X_{i:n},X;t)
\ge
\frac{r(t)}{2}
\left(
\eta(t)m_{i:n}(t)-1
\right).
\]

This proves part (i). Part (ii) follows similarly by reversing the above
inequalities when $f$ is convex.
\end{proof}
%
\begin{proposition}\label{prop:VarJ}
Let $X$ be a nonnegative absolutely continuous random variable with
continuously differentiable density function $f$.
Then, the derivative of the dynamic residual inaccuracy variance is given by
\begin{align}\label{e3.6}
\frac{d}{dt}\operatorname{VarJ}(X_{i:n},X;t)=\operatorname{VarJ}(X_{i:n},X;t)\Bigl(r_{i:n}(t)+2r(t)\Bigr)
-\frac14\,r_{i:n}(t)\Bigl(r(t)+2J(X_{i:n},X;t)\Bigr)^2,
\end{align}
where $r(t)$ and $r_{i:n}(t)$ denote the hazard rate functions of $X$ and $X_{i:n}$, respectively.
\end{proposition}
\begin{proof}
From the Equation \eqref{v3.1}, we have
\[\operatorname{VarJ}(X_{i:n},X;t)=\frac{1}{4(\overline F(t))^2\overline F_{i:n}(t)}\int_{t}^{+\infty}f_{i:n}(x)f^2(x)dx-\left(J(X_{i:n},X;t)\right)^2,\]
Differentiating both sides with respect to $t$ and using
\begin{align}
\frac{d}{dt}J(X_{i:n},X;t)
=\frac12\,r_{i:n}(t)r(t)+J(X_{i:n},X;t)\Bigl(r_{i:n}(t)+r(t)
\Bigr),
\end{align}
after straightforward algebraic manipulations, the desired result follows.
\end{proof}
\begin{corollary}
Suppose that the assumptions of Theorem~\ref{thm:BoundsJ}(i) hold.
Furthermore, assume that
\[
r(t)+2J(X_{i:n},X;t)\ge0.
\]
If
\[
\operatorname{VarJ}(X_{i:n},X;t)
\le
\frac{
r_{i:n}(t)\eta^{2}(t)r^{2}(t)m_{i:n}^{2}(t)
}
{4\left(r_{i:n}(t)+2r(t)\right)},
\]
then
\[
\frac{d}{dt}\operatorname{VarJ}(X_{i:n},X;t)\le0.
\]
Consequently,
\[
\operatorname{VarJ}(X_{i:n},X;t)
\]
is a decreasing function of \(t\).
\end{corollary}
Motivated by the characterization of  dynamic residual inaccuracy, we next investigate whether its associated variance can also characterize the underlying lifetime distribution.  
%
\begin{theorem}
Let $X$ be a nonnegative absolutely continuous random variable with
cumulative distribution function $F$. Suppose that
\begin{equation}\label{eq3.8}
J(X_{1:n},X;t)=C_2r(t),
\end{equation}
and
\begin{equation}\label{eq3.9}
V(t)=\operatorname{VarJ}(X_{1:n},X;t)=C_1r^2(t),
\end{equation}
where $C_2<0$ and $C_1>0$ are constants. Then $X$ belongs to one of the
following classes:
\begin{enumerate}
\item $X$ has an exponential distribution if and only if
\begin{equation}
C_1=\frac{n(1+2C_2)^2}{4(n+2)}.
\end{equation}
\item $X$ has a finite-range distribution if and only if
\begin{align}
    C_1>\frac{n(1+2C_2)^2}{4(n+2)}.
\end{align}

\item $X$ has a Lomax distribution if and only if
\begin{align}
    C_1<\frac{n(1+2C_2)^2}{4(n+2)}.
\end{align}
\end{enumerate}
\end{theorem}
\begin{proof}
We first establish the sufficiency part. 
By using equation \eqref{e3.6} with $n=1$, we obtain
\begin{align}\label{eq*}
V'(t)=V(t)(r_{1:n}(t)+2r(t))-\frac{1}{4}r_{1:n}(t)(r(t)+2J(X_{1: n}, X; t))^2
\end{align}
and also using 
\begin{align*}
r_{1:n}(t)=nr(t),
\end{align*}
and equations \eqref{eq3.8} and \eqref{eq3.9}, equation  \eqref{eq*} reduce to

\begin{align*}
2C_1r'(t)r(t)=C_1(n+2)r^3(t)-\frac{n}{4}(1+2C_2)^2r^3(t).
\end{align*}
Hence,
\begin{equation*}
\frac{2C_1r'(t)}{r^2(t)}=C_1(n+2)-\frac{n}{4}(1+2C_2)^2,
\end{equation*}
Now, if  $\frac{r'(t)}{r^2(t)}=A, \, \, t\geq 0,$ where 
$A=\frac{C_1(n+2)-\frac{n}{4}(1+2C_2)^2}{2C_1}$, then 
integrating the above differential equation gives $r(t)=\frac{1}{q-At}$, where $q=\frac{1}{r(0)}.$

If $C_1(n+2)=\frac{n}{4}(1+2C_2)^2,$ then $A=0$, and consequently $r'(t)=0.$

Thus, $r(t)=\lambda$ for some $\lambda>0$, which implies that the
hazard rate is constant. Therefore, $X$ has an exponential distribution.
Next, suppose that
\begin{equation}
C_1(n+2)
>
\frac{n}{4}(1+2C_2)^2.
\end{equation}
Then $A>0$, and
\begin{equation}
r(t)=\frac{1}{q-At},
\qquad
0\leq t<\frac{q}{A}.
\end{equation}
Since
$\frac{d}{dt}\left[-\log\overline{F}(t)\right]=\frac{1}{q-At},$
we obtain
\begin{align}
\overline{F}(t)
=\exp\left\{-\int_0^t\frac{dx}{q-Ax}\right\} =\left(1-\frac{A}{q}t\right)^{1/A},
\qquad
0\leq t<\frac{q}{A}.
\end{align}
Hence, the survival function vanishes at the finite endpoint $q/A$, and
$X$ has a finite-range distribution.
Finally, if
\begin{equation}
C_1(n+2)<\frac{n}{4}(1+2C_2)^2,
\end{equation}
then $A<0$. Writing $A=-p>0$, gives $r(t)=\frac{1}{q+p t}.$
Consequently,
\begin{align}
\overline{F}(t)=\exp\left\{-\int_0^t\frac{dx}{q+p x}\right\}=\left(1+\frac{p}{q}t\right)^{-1/p},
\end{align}
With considering $\frac{q}{p}=\lambda$ and $\frac{1}{p}=\alpha$, we have $f(t)=\frac{\alpha}{\lambda}(1+\frac{t}{\lambda})^{-(1+\alpha)}$, for $t\ge 0$,
which is the probability distribution function of a Lomax distribution. 

The necessity of Parts 1--3 can be established through straightforward calculations. 
For example, if $X$ has a Lomax distribution, then simplify we can compute
\begin{align*}
r(t)=\frac{\alpha}{\lambda}\left(1+\frac{t}{\lambda}\right)^{-1},
\end{align*}
\begin{align*}
J(X_{1: n}, X; t)&=-\frac{n\alpha}{2(\alpha(n+1)+1)}r(t),
\end{align*}
and
\begin{align*}
\operatorname{VarJ}(X_{1: n}, X; t)=\left[\frac{n\alpha}{4(\alpha(n+2)+2)}-\frac{n^2\alpha^2}{4(\alpha(n+1)+1)^2}\right]r^2(t),
\end{align*}
 and 
hence, the proof is complete.
\end{proof}

%
%
We next investigate whether the hazard rate ordering is preserved by the dynamic residual measure of inaccuracy based on extropy and varextropy for the Lomax distribution.
\begin{theorem}
Let $X\sim \mathrm{Lomax}(\alpha_1,\lambda),$ and $Y\sim \mathrm{Lomax}(\alpha_2,\lambda),$
where $\lambda>0$ is fixed. If $\alpha_1\ge \alpha_2,$ then $X\le_{hr}Y.$
Furthermore, $J(X_{1:n},X;t)\le J(Y_{1:n},Y;t), t\ge0.$
\end{theorem}
\begin{proof}
Since for $t\ge0$, $r_X(t)=\frac{\alpha_1}{\lambda+t}$ and $r_Y(t)=\frac{\alpha_2}{\lambda+t},$
it follows immediately that $r_X(t)\ge r_Y(t)$ and hence $X\le_{hr}Y.$

On the other hand, 
$J(X_{1:n},X;t)=-\frac{n\alpha^2}{2(\alpha(n+1)+1)(\lambda+t)}=:H(\alpha).$

Differentiating $H(\alpha)$ with respect to $\alpha$ gives
\[H'(\alpha)=-\frac{n\alpha\bigl(\alpha(n+1)+2\bigr)}{2(\lambda+t)\bigl(\alpha(n+1)+1\bigr)^2}<0,\]
for every $\alpha>0$. Therefore,
\[\alpha_1\ge\alpha_2\quad\Longrightarrow\quad H(\alpha_1)\le H(\alpha_2),\]
which proves that
\[J(X_{1:n},X;t)\le J(Y_{1:n},Y;t).\]
\end{proof}
%
%
\begin{theorem}
Let $X\sim \mathrm{Lomax}(\alpha,\lambda_1)$ and $Y\sim \mathrm{Lomax}(\alpha,\lambda_2)$,
where $\alpha>0$ is fixed. If $\lambda_1\le \lambda_2,$
then $X\le_{hr}Y.$
Moreover,
\(J(X_{1:n},X;t)\le J(Y_{1:n},Y;t),\)
and
\[\operatorname{VarJ}(X_{1:n},X;t)
\ge
\operatorname{VarJ}(Y_{1:n},Y;t),
\qquad t\ge0.
\]
\end{theorem}
\begin{proof}
Since \(r_X(t)=\frac{\alpha}{\lambda_1+t}$ and $ r_Y(t)=\frac{\alpha}{\lambda_2+t},\) for $t\ge0$,
we have \(r_X(t)\ge r_Y(t),\)
which implies $X\le_{hr}Y.$

Furthermore, \(J(X_{1:n},X;t)=-\frac{n\alpha}{2(\alpha(n+1)+1)}\,r_X(t),\)
where $-\frac{n\alpha}{2(\alpha(n+1)+1)}<0.$
Hence,
$J(X_{1:n},X;t)\le J(Y_{1:n},Y;t).$

Also,
\[\operatorname{VarJ}(X_{1:n},X;t)=c_1(\alpha)\, r_X^2(t),\]
where
\[c_1(\alpha)=\frac{n\alpha}{4(\alpha(n+2)+2)}
-\frac{n^2\alpha^2}{4(\alpha(n+1)+1)^2}
>0.\]
Therefore,
\[\operatorname{VarJ}(X_{1:n}, X; t)
\ge\operatorname{VarJ}(Y_{1:n},Y;t).\]
\end{proof}
\begin{remark}
Suppose that $X\sim \mathrm{Lomax}(\alpha_1,\lambda_1)$ and  $Y\sim \mathrm{Lomax}(\alpha_2,\lambda_2),$
with both the shape and scale parameters allowed to vary.
Although the conditions $\alpha_1\ge\alpha_2$ and $\lambda_1\le\lambda_2$ guarantee that $X\le_{hr}Y,$
they do not, in general, imply an ordering of either $J(X_{1:n},X;t)$ and $\operatorname{VarJ}(X_{1:n},X;t)$ relative to their counterparts for $Y$. 
 This is due to the explicit dependence of the coefficients multiplying  the hazard rate on the shape parameter~$\alpha$.
\end{remark}
\section{Nonparametric  estimation }\label{Section4}
In this section, we propose a nonparametric estimation of $\operatorname{VarJ}(X_{i:n}, X;  t).$
Let \(X_1, \ldots, X_n\) be a random sample from a population with
probability density function (PDF) \(f(\cdot)\) and cumulative
distribution function (CDF) \(F(\cdot)\). 
Based on the kernel estimation framework introduced by Silverman \cite{Silverman},
the probability density function \(f(\cdot)\) is estimated as
\[\widehat{f}(x)=\frac{1}{n h_f}\sum_{i=1}^{n}K\left(\frac{x-X_i}{h_f}\right),\]
where \(K(\cdot)\) denotes the kernel function and \(h_f\) is the
corresponding bandwidth parameter.

The bandwidth parameter plays a crucial role in kernel density
estimation because it determines the amount of smoothing in the
estimated density. Small bandwidths tend to produce highly variable
and undersmoothed estimates, whereas large bandwidths result in
smoother estimates that may fail to capture important features of the
underlying distribution. Hence, the selection of an appropriate
bandwidth is an important step in kernel density estimation.

In this study, the least-squares cross-validation (CV) method is used
to select the bandwidth parameter. The leave-one-out CV bandwidth is
obtained by minimizing the least-squares cross-validation criterion
and is defined as
\[
h_f^{\mathrm{CV}}
=
\underset{h_f>0}{\arg\min}
\left\{
\int_{-\infty}^{\infty}
\widehat{f}_{h_f}^{\,2}(x)\,dx
-
\frac{2}{n}
\sum_{i=1}^{n}
\widehat{f}_{h_f,-i}(X_i)
\right\},
\]
where \(\widehat{f}_{h_f,-i}(X_i)\) denotes the kernel density
estimator evaluated at \(X_i\), with the \(i\)-th observation omitted
from the sample.

For the estimation of the cumulative distribution function, we employ the kernel-based approach proposed by Nadaraya \cite{Nadaraya}. The kernel estimator of the CDF is given by
\begin{align}
\widehat{F}_{h_F}(x)=\frac{1}{n}\sum_{i=1}^{n}W\left(\frac{x-X_i}{h_F}\right),
\end{align}
where $W(x)=\int_{-\infty}^{x} K(t)\,dt$
denotes the cumulative distribution function associated with the
kernel function \(K(\cdot)\). In the present study, the standard
normal kernel is considered, for which \(W(\cdot)=\Phi(\cdot)\).

The bandwidth $h_F$ is also selected using a cross-validation
criterion. Specifically, the leave-one-out CV bandwidth is determined
by minimizing the integrated squared error between the empirical
indicator function and the leave-one-out kernel estimator, given by
\[h_F^{\mathrm{CV}}
=
\underset{h_F>0}{\arg\min}
\frac{1}{n}
\sum_{i=1}^{n}
\int_{-\infty}^{+\infty}
\left[
I(x-X_i\geq 0)
-
\widehat{F}_{h_F,-i}(x)
\right]^2
dx,
\]
where
\[
\widehat{F}_{h_F,-i}(x)
=
\frac{1}{n-1}
\sum_{\substack{j=1\\j\neq i}}^{n}
\Phi\left(
\frac{x-X_j}{h_F}
\right)
\]
denotes the leave-one-out kernel estimator of the CDF (see Bowman et al.  \cite{Bowman}).

Using the kernel-based estimators of \(f(\cdot)\) and \(F(\cdot)\),
the density and distribution function of the minimum order statistic
\(X_{1:n}\) are estimated as
\[\widehat{f}_{1:n}(x)=n\left[1-\widehat{F}(x)\right]^{n-1}\widehat{f}(x),\]
and
\[\widehat{F}_{1:n}(t)=1-\left[1-\widehat{F}(t)\right]^n.\]

Consequently, the nonparametric estimator of the dynamic residual extropy-based variance measure is given by
\[
\widehat{\operatorname{VarJ}}(X_{1:n},X;t)
=
\frac{1}{4}
\left\{
\int_t^\infty
\left[
\frac{\widehat{f}(x)}
{\widehat{\overline{F}}(t)}
\right]^2
\frac{\widehat{f}_{1:n}(x)}
{\widehat{F}_{1:n}(t)}
\,dx
-
\left[
\int_t^\infty
\frac{\widehat{f}(x)}
{\widehat{\overline{F}}(t)}
\frac{\widehat{f}_{1:n}(x)}
{\widehat{F}_{1:n}(t)}
\,dx
\right]^2
\right\},
\]
where
\[\widehat{\overline{F}}(t)=1-\widehat{F}(t).\]
\begin{figure}[ht]
\begin{center}
\includegraphics[width=7cm,height=6cm]{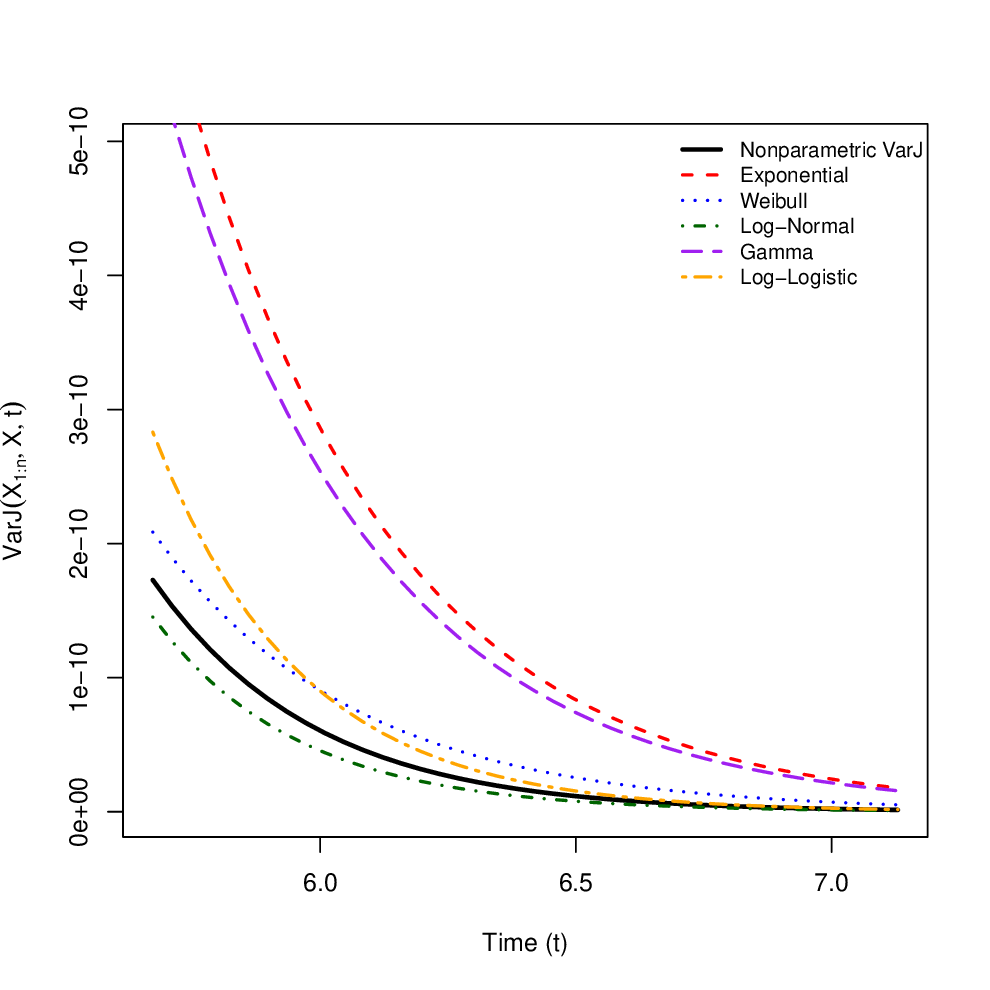}\\
\vspace{-0.2cm}
\caption{\footnotesize{Estimation of $\mathrm{VarJ}(X_{1:n},X;t)$ for the number of casualties data }}\label{fig1}
\end{center}
\end{figure}

\section{Real  data}\label{Section5}
To demonstrate the practical applicability of the proposed measure, we consider a real data set consisting of the numbers of casualties
reported in $n=44$ different plane crashes, originally considered in \cite{Chowdhury}.  Following the analysis in \cite{Moham}, the exponential, Weibull, gamma, log-normal, and log-logistic distributions were considered as candidate models for the data.

In the present study, the proposed $\mathrm{VarJ}(X_{1:n},X;t)$ measure is employed to assess the
agreement between the nonparametric and parametric representations of the data. The nonparametric estimate of
$\mathrm{VarJ}(X_{1:n},X;t)$ is obtained from the kernel-based estimators of the PDF and CDF introduced in the previous section.
The corresponding theoretical values are then calculated for each of the considered models.

Figure~\ref{fig1} presents the nonparametric estimate of $\mathrm{VarJ}(X_{1:n},X;t)$, together with the corresponding
theoretical curves for different values of $t$. As shown in the figure, the theoretical curve associated with the log-normal model
exhibits the closest agreement with the nonparametric estimate over the considered range of $t$. This result indicates that, among the
considered models, the log-normal distribution provides the most suitable representation of the observed data in terms of the
proposed $\mathrm{VarJ}$ measure.

This finding is consistent with the results reported in \cite{Moham}, where
the log-normal distribution was identified as an appropriate model
for these data. Therefore, the proposed
\(\mathrm{VarJ}(X_{1:n},X;t)\) measure provides an additional
distributional assessment based on the dynamic residual variability
structure of the data and offers a useful tool for comparing
competing parametric models.
\section{Conclusion}
In this paper, we introduced a measure for quantifying the dispersion of residual inaccuracy associated with an order statistic. We showed that, for the exponential distribution, this measure is symmetric with respect to $i$. When a closed-form expression for the proposed measure is not available, deriving suitable upper and lower bounds can be useful. To this end, we obtained bounds for the proposed measure using Chernoff’s inequality. We also derived bounds in terms of the hazard rate and the eta function. We established that the proposed variance-based inaccuracy measure is invariant under location transformations but not under scale transformations. Furthermore, a dynamic residual version of the measure was developed, together with characterization results and bounds based on extropy and varextropy. A nonparametric kernel estimator of $\operatorname{VarJ}(X_{i:n},X;t)$ was also provided. The cross-validation (CV) method was considered for bandwidth selection in both PDF and CDF kernel estimation. Finally, an application to a real dataset demonstrated how the proposed measure can be used in model selection. 
\bibliographystyle{amsplain}

\end{document}